\documentclass[11pt,oneside]{amsart}
\usepackage[utf8]{inputenc} 
\usepackage[margin=1in]{geometry}
\usepackage{amscd}
\usepackage{hyperref}
\usepackage{fontenc}
\usepackage{textcomp}
\usepackage{comment}
\theoremstyle{plain}
\usepackage{enumitem}
\usepackage{amssymb,amsmath,amsthm}
\usepackage{mathtools}
\usepackage{mathrsfs}
\usepackage{xcolor}
\usepackage[all]{xy}
\usepackage{cancel}
\usepackage{graphicx}
\usepackage{lscape}
\usepackage{cite}
\usepackage{mathabx}

\newtheorem*{thmnonnum}{Theorem}

\theoremstyle{main}

\theoremstyle{main}

\newtheorem{lem}{Lemma}
\newtheorem{rema}{Remark}

\newtheorem{fact}{Fact}

\newtheorem*{remer}{Acknowledgements}
\newtheorem{defi}{Definition}

\newtheorem{coro}{Corollary}

\newtheorem{nota}{Notation}

\newcommand\nn{\mathbb{N}}

\newcommand\cc{\mathbb{C}}

\newcommand\tq{\mbox{ } | \mbox{ }}

\newcommand\cali[1]{\mathcal{#1}}

\newcommand\hilb{\mathcal{H}}

\DeclareMathOperator{\id}{Id}

\DeclareMathOperator{\fix}{Fix}

\title{On a generalization of the Howe-Moore property}

\author{Antoine Pinochet Lobos}\thanks{Aix-Marseille Université, a.p.lobos@outlook.com}

\begin{document}

\maketitle

\renewcommand{\abstractname}{Abstract}
\begin{abstract}
We define a Howe-Moore property relative to a set of subgroups. Namely, a group $G$ has the Howe-Moore property relative to a set $\mathcal{F}$ of subgroups if for every unitary representation $\pi$ of $G$, whenever the restriction of $\pi$ to any element of $\mathcal{F}$ has no non-trivial invariant vectors, the matrix coefficients vanish at infinity. We prove that a semisimple group has the Howe-Moore property relatively to the family of its factors.
\end{abstract}

\section{Introduction}

In \cite{HOWEMOORE}, Howe and Moore discovered a very interesting property of connected, non-compact, simple Lie groups with finite center: whenever they act ergodically on a probability space by preserving the measure, the action is automatically mixing. This property, rephrased purely in terms of unitary representations has since been called the \textbf{Howe-Moore} property. Later, other topological groups were proved to enjoy this property.

In \cite{CIOBO}, a very beautiful paper, Ciobotaru synthesizes the proofs of all known cases of groups having the Howe-Moore property, giving a unified proof.

In this paper, we generalize further the unified proof of \cite{CIOBO} so that it also applies to products and, in particular, generalizes the situation of products of Lie groups considered in \cite[Theorem 1.1, p. 81]{BEKKAMAYER}).

\section{Statement of the results}

Let $G$ be a topological group.

\begin{nota}
If $g \in G^\nn$, let us write $\lim_{n \to \infty} g_n = \infty$ if for every compact subset $K$ of $G$, there is an integer $N$ such that for any integer $n$ such that $n \geq N$, $g_n \not \in K$.

If $f : G \rightarrow \cc$, if $a \in \cc$, we write $\lim_{g \to \infty} f(g) = a$ when we have \[\forall \epsilon > 0,\ \exists K \subset G,\quad K\mbox{ is compact and}\ \forall g \not \in K,\ \vert f(g) - a \vert \leq \epsilon.\]
\end{nota}

\begin{defi}\textbf{(Cartan decomposition)} 

We say that a triplet $(K_1,A^+,K_2)$ is a \textbf{Cartan decomposition} of $G$ if the following conditions are satisfied:
\begin{enumerate}
\item $K_1$ and $K_2$ are compact subsets of $G$,
\item $A^+$ is an abelian subsemigroup $G$, that is, $\forall a_1,a_2 \in A^+,\ a_1 a_2 = a_2 a_1 \in A^+$ and
\item $G = K_1A^+K_2$.
\end{enumerate}
\end{defi}

\begin{nota}
If $a \in G^\nn$, we set \[U^+_{a} := \{g \in G \tq \lim_{n \to \infty} a_n^{-1}\ g\ a_n = e\}\ \mbox{et}\] \[U^-_{a} := \{g \in G \tq \lim_{n \to \infty} a_n\ g\ a^{-1}_n = e\}.\]We call them the \textbf{the positive and negative contracting subgroups} associated to $a$.
\end{nota}

\begin{defi}\textbf{(Mautner's property)}

Let $\cali{F}$ be a set of subgroups of $G$, and $A$ a subset of  $G$. We say that $(G,A)$ has \textbf{Mautner's property} relative\footnotemark[1] to $\cali{F}$ if \[\forall a \in A^\nn\ \left(\lim_{n \to \infty} a_n = \infty\right) \Longrightarrow \left(\exists F \in \cali{F},\ \exists b \ \mbox{subsequence of}\ a, \quad F \subseteq \overline{\langle U^+_{\mathbf{b}},U^-_{\mathbf{b}}\rangle}\right).\]
\end{defi}

\footnotetext[1]{If $\cali{F} = \{G\}$, we omit ``relative to $\cali{F}$".}

\begin{rema} In \cite{CIOBO}, it is proved that the following groups have Cartan decompositions $(K_1,A^+,K_2)$ such that $(G,A^+)$ has the Mautner property:
\begin{enumerate}
\item simple algebraic groups over a non-archimedean local field;
\item subgroups of the group of automorphisms of a $d$-biregular tree for $d\geq 3$ that are topologically simple and that act $2$-transitively on the boundary of the tree;
\item noncompact, connected, semisimple Lie groups with a finite center.
\end{enumerate}
\end{rema}

\begin{nota}
If $\pi : G \rightarrow U(\cali{H})$ is a unitary representation of $G$ and $F$ is a subgroup of $G$. We denote by\[\fix(\pi,F) := \{\phi \in \cali{H} \tq \forall g \in F,\quad \pi(g)\phi = \phi\}.\]
\end{nota}

\begin{defi}\textbf{(Relative Howe-Moore property)}

Let $\cali{F}$ be a set of subgroups of $G$. We say that $G$ has the \textbf{Howe-Moore property} relative\footnotemark[2] to $\cali{F}$ if \[\forall \pi : G \rightarrow U(\cali{H}),\quad \left(\forall F \in \cali{F},\quad \fix(\pi,F) = \{0\}\right) \Longrightarrow \left(\forall \phi,\psi \in \cali{H},\quad \lim_{g \to \infty} \langle \pi(g)\phi,\psi\rangle = 0\right).\]
\end{defi}

\footnotetext[2]{As for Mautner's property, if $\cali{F} = \{G\}$, we omit ``relative to $\cali{F}$".}

\begin{rema} In \cite{CLUCORLOUTESVAL}, one can find a ``relative Howe-Moore property", but the one we consider in the present note is different.
\end{rema}

Our main result is the following.

\begin{thmnonnum} Let $\cali{F}$ be a set of subgroups of $G$. If $G$ admits a Cartan decomposition $(K_1,A^+,K_2)$ such that $(G,A^+)$ has the Mautner property relative to $\cali{F}$, then it satisfies the Howe-Moore property, relative to $\cali{F}$.
\end{thmnonnum}

\begin{rema} In the case where $\cali{F} = \{G\}$, then the theorem is just \cite[Theorem 1.2, p. 2]{CIOBO}.
\end{rema}

The following consequence is useful.

\begin{coro}\label{corollaire1} Let $G_1,...,G_N$ be groups having Cartan decompositions $(K_{i,1},A^+_i,K_{i,2})$ such that for all $i$, $(G_i,A^+_i)$ has the Mautner property. Then the product $G:= \Pi_i G_i$ has the Howe-Moore property, relative to $\{G_1,\cdots ,G_N\}$.
\end{coro}

The Howe-Moore property is often used to deduce mixing from ergodicity. The following obvious corollary states the analog result for the relative Howe-Moore property.

\begin{coro}\label{corollaire2} Let $G_1, \cdots, G_N$ be topological groups admitting Cartan decompositions $(K_{i,1},A^+_i,K_{i,2})$ such that for all $i$, $(G_i,A^+_i)$ has the Mautner property. Let $G := G_1 \times \cdots \times G_N$, and let $G \curvearrowright (X,\mu)$ a measure-preserving action on a probability space such that the restriction to each of the $G_i's$ is ergodic. Then the action $G \curvearrowright (X,\mu)$ is mixing.
\end{coro}

As an application, we spell out the following corollary.

\begin{coro}\label{reseauirreductibleergodicite} Let $G_1,\cdots, G_N$ be topological groups having Cartan decompositions $(K_{i,1},A^+_i,K_{i,2})$ such that for each $i$, $(G_i,A^+_i)$ has the Mautner property. Let $G := G_1 \times \cdots \times G_N$, and let $\Gamma$ be an irreducible lattice $G$. Then the action $G \curvearrowright G/\Gamma$ is mixing.
\end{coro}

\begin{rema} The theorem and Corollary \ref{corollaire2} were already known, in the case $G$ is a semisimple group with finite center (see \cite[Theorem 1.1, p. 81 and Theorem 2.1, p. 89]{BEKKAMAYER} for a proof using Lie theory technology). In the approach we propose, Lie theory is only needed to prove that the factors satisfy the Howe-Moore property. We therefore provide an elementary shortcut for a part of their proof. Moreover, our proof is more general and applies to other topological groups.
\end{rema}

\begin{remer} We would like to adress many thanks to Christophe Pittet for his useful help and advice.
\end{remer}

\section{Proofs}

\subsection{Useful facts and notation}

We recall here some tools we need for the proofs.

\begin{nota}
A sequence $g \in G^\nn$ is said to be \textbf{bounded} if $\exists K,\ \forall n\in \nn, \quad g_n \in K$.
\end{nota}

\begin{fact}
If $G$ is locally compact, second countable, every unbounded sequence has a subsequence that goes to infinity.

If $G$ is locally compact, second countable, $f : G \rightarrow \cc$ and $a \in \cc$, we have the following sequential characterization \[\displaystyle\lim_{g \to \infty} f(g) = a \Leftrightarrow \left(\forall g \in G^\nn, \lim_{n \to \infty} g_n = \infty \Rightarrow \lim_{n \to \infty} f(g_n) = a\right).\]
\end{fact}

When we write $\pi : G \rightarrow U(\hilb)$, it is implicit that it is both a morphism and that it is continuous for the strong operator topology (and therefore, a \textbf{unitary representation}), and that $\hilb$ is a complex, separable Hilbert space.

The following fact easily follows from the sequential weak operator compactness of the unit ball in the space of bounded operators. 

\begin{fact}\label{limitcommutnormal}

Let $T = (T_n)_{n \in \nn}$ be a sequence of normal operators of norm $1$ on a Hilbert space such that $\forall n,m\in \nn, T_n T_m = T_m T_n$. Then $T$ has a subsequence, that converges, in the weak operator topology, to a normal operator that commutes with all the $T_n's$.
\end{fact}

\subsection{Proof of the theorem}

The proof of the following lemma is obvious.

\begin{lem}\label{ssgroupcontractproduit} If $a = (a^1,\cdots,a^N) \in (G_1 \times \cdots \times G_N)^\nn$, then $U^+_a = U^+_{a^1} \times \cdots \times U^+_{a^N}$ (and this is also valid for $U^-$).
\end{lem}

Let $\pi : G \rightarrow U(\hilb)$ be a unitary representation.

\begin{lem}{\rm \cite[Lemma 2.9]{CIOBO}}
\label{coefficientmautner}
Let $(K_1,A^+,K_2)$ be a Cartan decomposition of $G$. If \[
\exists \phi,\psi \in \hilb\setminus \{0\},\ \exists g \in G^\nn,\quad (\langle  \pi(g_n)\phi,\psi\rangle)_{n \in \nn} {\rm\ doesn't\ converge\ to\ vers}\ 0,\]then \[\exists \phi,\psi \in \hilb\setminus \{0\},\ \exists a \in (A^+)^\nn,\quad (\langle \pi(a_n)\phi,\psi\rangle)_{n \in \nn} {\rm\ doesn't\ converge\ to\ vers}\ 0.\]
\end{lem}

\begin{lem}{\rm \cite[Lemma 2.8]{CIOBO}}
\label{coefficientdiagonal}
Let $g \in G^\nn$. If \[\exists \phi,\psi \in \hilb\setminus \{0\},\quad (\langle \pi(g_n)\phi,\psi\rangle)_{n \in \nn} {\rm\ doesn't\ converge\ to\ vers}\ 0,\]then
\[\exists \phi\in \hilb\setminus \{0\},\quad (\langle \pi(g_n)\phi,\phi\rangle)_{n \in \nn} {\rm\ doesn't\ converge\ to\ vers}\ 0.\]
\end{lem}

\begin{lem}\label{ssgroupefermefixe} If $\phi \in \hilb$, the set $\{g \in G \tq \pi(g) \phi = \phi\}$ is a closed subgroup.
\end{lem}

\begin{proof}[Proof] It is a subgroup because $\pi$ is a morphism, and it is closed because $\pi$ is strongly continuous.
\end{proof}

We extract the following lemma out of \cite[Lemma 3.1]{CIOBO} for the sake of clarity.

\begin{lem}\label{mautnerplusplus} Let $g \in G^\nn$ such that $\forall n,m \in \nn, g_n g_m = g_m g_n$. Let $\phi \in \hilb\setminus\{0\}$ such that $(\langle \pi(g_n)\phi,\phi\rangle)_{n \in \nn}$ doesn't converge to $0$. Then there is $\phi_0 \in \hilb\setminus\{0\}$, fixed by $U^+_{g}$ and by $U^-_g$.
\end{lem}

\begin{proof}[Proof] Up to extraction, we can assume that $(\pi(g_n))_{n \in \nn}$ converges, for the weak operator topology, to a normal operator $E$, which commutes with the $\pi(g_n)'s$, according to Fact \ref{limitcommutnormal}.

Because of the weak operator convergence of the operators, $\langle E\phi,\phi\rangle \not = 0$, which implies that $E\phi \not = 0$. Let us prove that $E\phi$ is fixed by $U^\pm_g$.

Let $u \in U^+_g$, and $\psi \in \hilb$. We have
\[\begin{array}{rcl}
\vert \langle \pi(u)E\phi - E\phi,\psi \rangle \vert &= &\vert\langle E\pi(u)\phi - E\phi,\psi \rangle \vert\\
&= &\displaystyle\left\vert \lim_{n \to \infty} \langle\pi(g_n)\pi(u)\phi - \pi(g_n)\phi,\psi\rangle\right\vert\\
&= &\displaystyle\left\vert \lim_{n \to \infty} \langle\pi(g_nug^{-1}_n)\pi(g_n)\phi - \pi(g_n)\phi,\psi\rangle\right\vert\\
&= &\displaystyle\left\vert \lim_{n \to \infty} \left\langle\left(\pi(g_nug^{-1}_n) - \id\right)\pi(g_n)\phi,\psi\right\rangle\right\vert\\
&= &\displaystyle\left\vert \lim_{n \to \infty} \left\langle\pi(g_n)\phi,\left(\pi(g_nug^{-1}_n) - \id\right)^*\psi\right\rangle\right\vert\\
&= &\displaystyle\left\vert \lim_{n \to \infty} \left\langle\pi(g_n)\phi,\left(\pi(g^{-1}_nu^{-1}g_n) - \id\right)\psi\right\rangle\right\vert\\
&\leq &\lim_{n \to \infty} \Vert \pi(g_n)\phi \Vert \cdot \Vert \left(\pi(g^{-1}_nu^{-1}g_n) - \id\right)\psi\Vert\\ 
&= &\lim_{n \to \infty} \Vert \phi \Vert \cdot \Vert \left(\pi(g^{-1}_nu^{-1}g_n) - \id\right)\psi\Vert\\
{\small \mbox{($u^{-1} \in U^+_G$)}}&\to &0\\
\end{array}\]This being true for all $\psi$, we therefore have $\pi(u)E\phi = E\phi$. We use the same procedure to prove that $u \in U^-_g$.
\end{proof}

\begin{proof}[Proof of the theorem]

Let us prove that if there is $\phi, \psi \in \hilb$ such that we don't have \[\lim_{g \to \infty} \langle \pi(g)\phi,\psi\rangle = 0,\] then there is $F \in \cali{F}$ and a vector $\phi_0 \in \hilb \setminus\{0\}$ fixed by $\pi(F)$.

So, let $\phi,\psi \in \hilb$ be as such. There is a sequence $g \in G^\nn$ that goes to infinity such that $(\langle\pi(g_n)\phi,\psi\rangle)_{n \in \nn}$ doesn't converge to $0$. Up to extraction, we can assume that there exists $F \in \cali{F}$ such that $F \subseteq \overline{\langle U^+_g,U^-_g\rangle}$. According to Lemma \ref{coefficientmautner} and Lemma \ref{coefficientdiagonal}, we can assume that $(\langle\pi(g_n)\phi,\phi\rangle)_{n \in \nn}$ doesn't converge to $0$, and that $g \in (A^+)^\nn$. According to Lemma \ref{mautnerplusplus}, there is $\phi_0 \in \hilb \setminus\{0\}$ that is fixed by $U^\pm_g$. According to Lemma \ref{ssgroupefermefixe}, $\phi_0$ is, in fact, fixed by $\overline{\langle U^+_g,U^-_g\rangle}$, and therefore, by $F$.
\end{proof}

\subsection{Proofs of the corollaries}

The proof of Corollary \ref{corollaire1} is an obvious application of the following lemma.

\begin{lem} Let $G_1,...,G_N$ be topological groups such that for each $i$, $G_i$ has a Cartan decomposition $(K_{i,1}, A^+_1,K_{i,2})$, and $(G_i,A^+_i)$ has the Mautner property. Then \[(K_{1,1} \times \cdots \times K_{n,1}, A^+_1 \times \cdots \times A^+_n, K_{1,2} \times \cdots \times K_{n,2})\] is a Cartan decomposition of $G := G_1 \times \cdots \times G_N$ such that $(G,A^+_1 \times \cdots \times A^+_n)$ has the Mautner property relative to $\{G_1,...,G_N\}$.
\end{lem}

\begin{proof}[Proof] It is clear that the announced triplet is a Cartan decomposition of $G$. We just have to prove that $(G,A^+_1\times \cdots \times A^+_N)$ satisfies Mautner's property, relative to $\{G_1,...,G_N\}$. Let us denote $A^+ := A^+_1 \times \cdots \times A^+_N$. Let $a = (a^1,\cdots, a^N) \in (A^+)^\nn$ such that $\lim_{n \to \infty} a_n = \infty$. The set $\{i \in \{1,...,N\} \tq (a^i_n)_{n \in \nn}\ {\rm is\ not\ bounded}\}$ is not empty, unless $a$ is itself bounded, but it isn't by hypothesis. Let $j$ be such that $(a^j_n)_{n \in \nn}$ is unbounded. Then there is an increasing $h_1 : \nn \rightarrow \nn$ such that $(a^j_{h_1(n)})_{n \in \nn}$ goes to infinity in $G_j$. By hypothesis on $G_j$, there is an increasing $h_2 : \nn \rightarrow \nn$ such that if we denote $b^j := (a^j_{h_1(h_2(n))})_{n \in \nn}$, then $\overline{\langle U^+_{b^j},U^-_{b^j}\rangle } = G_j$. We then have, by Lemma \ref{ssgroupcontractproduit} \[\overline{\langle U^+_{b},U^-_{b}\rangle } \supseteq \{1\}\times \cdots \times \{1\} \times G_j \times \{1\} \times \cdots \times \{1\}.\]
\end{proof}

To a measure preserving action of a topological group on a probability space $X$, one can associate a unitary representation of the group in $U(L^2(X))$ (called the Koopman representation) such that the action is ergodic if and only if the only invariant vectors of the representation are the constants, and such that the mixing is equivalent to the vanishing at infinity of all matrix coefficients of the subrepresentation on the subspace of functions of zero integral. This said, the proof of Corollary \ref{corollaire2} is obvious.

The proof of Corollary \ref{reseauirreductibleergodicite} goes as follows.

\begin{proof}[Proof of Corollary \ref{reseauirreductibleergodicite}] Thanks to Corollary \ref{corollaire2}, it is enough to check that for every $i$, $G_i \curvearrowright G/\Gamma$ is ergodic. According to \cite[Corollary 2.2.3, p. 18]{ZIMMER}, $G_i \curvearrowright G/\Gamma$ is ergodic if and only if $\Gamma \curvearrowright G/G_i$ is ergodic. But this action is ergodic if and only if the image of $\Gamma$ in $G_1\times \cdots \times \widehat{G_i} \times \cdots \times G_N$ is dense, and this is precisely the case when $\Gamma$ is irreducible.
\end{proof}

\bibliographystyle{alpha}
\bibliography{bibliord}

\end{document}